\documentclass[12pt]{amsart}

\usepackage{amsmath,amssymb,amsthm}
\usepackage[dvipdfm]{graphicx}
\usepackage{enumerate}
\usepackage[dvips]{color}
\usepackage{version}

\newtheorem{Thm}{Theorem}[section]
\newtheorem{Lem}[Thm]{Lemma}
\newtheorem{Fac}[Thm]{Fact}
\newtheorem{Cor}[Thm]{Corollary}
\newtheorem{Prop}[Thm]{Proposition}
\newtheorem{Conj}[Thm]{Conjecture}

\theoremstyle{remark}
\newtheorem{Rem}[Thm]{Remark}
\newtheorem{Ex}[Thm]{Example}
\theoremstyle{definition}

\newcommand{\lct}{\mathop{\mathrm{lct}}\nolimits}

\newcommand{\Aut}{\mathop{\mathrm{Aut}}\nolimits}
\newcommand{\SO}{\mathop{\mathrm{SO}}\nolimits}
\newcommand{\SL}{\mathop{\mathrm{SL}}\nolimits}

\newcommand{\PGL}{\mathop{\mathrm{PGL}}\nolimits}

\newcommand{\Supp}{\mathop{\mathrm{Supp}}\nolimits}

\newcommand{\DF}{\mathop{\mathrm{DF}}\nolimits}
\newcommand{\Sesh}{\mathop{\mathrm{Sesh}}\nolimits}

\begin{document}

\title[K-stability of $\mathbb{Q}$-Fano varieties ]
{Alpha invariant and K-stability of $\mathbb{Q}$-Fano varieties}
\author{Yuji Odaka}
\address{Research Institute for Mathematical Sciences, , Kyoto University, Kyoto 606-8502 Japan}
\email{yodaka@kurims.kyoto-u.ac.jp}

\author{Yuji Sano}
\address{Graduate School of Science and Technology, Kumamoto University,  2-39-1, Kurokami, Kumamoto, 860-8555, Japan
\\
Tel.: 0011-81-96-342-3333 $\quad$ Fax: 0011-81-96-342-3341
}
%\date{27th November 2010 }
\email{sano@sci.kumamoto-u.ac.jp}
\maketitle

\begin{abstract}
We give a purely algebro-geometric proof that if the $\alpha$-invariant of a $\mathbb{Q}$-Fano variety $X$ is greater than $\dim X/(\dim X+1)$, then 
$(X,\mathcal{O}_{X}(-K_X))$ is K-stable. 
The key of our proof is a relation among the Seshadri constants, the $\alpha$-invariant and K-stability. It also gives applications concerning the automorphism group. 
\end{abstract}

% \linenumbers

\section{Introduction }
The $\alpha$-invariant is introduced by Tian \cite{Tia87} to give a numerical criterion for the existence of K\"ahler-Einstein metrics on Fano manifolds.
On the other hands, it is conjectured that the existence of K\"ahler-Einstein metrics would be equivalent to K-stability of manifolds which is a certain version of stability notion of Geometric Invariant Theory. 
The purpose of this paper is to study a direct relation between the $\alpha$-invariant and K-stability from algebro-geometric viewpoint and give some applications.

Let $X$ be an $n$-dimensional smooth Fano manifold. 
We take into account a compact sub Lie group $G$ (possibly trivial) of the holomorphic automorphism group $\Aut(X)$. 
Let $\omega$ be a fixed $G$-invariant K\"ahler form with K\"ahler class $c_{1}(X)$. 
Let $P_G(X,\omega)$ be the set of K\"ahler potentials defined by
$$
	P_G(X,\omega)=\{\varphi\in C^2_\mathbb{R}(X) \mid G\mbox{-invariant},\,\sup\varphi=0,\,\omega+\frac{\sqrt{-1}}{2\pi}\partial\bar\partial \varphi >0\}.
$$
Tian \cite{Tia87} introduced the invariant
$$
	\alpha_G(X)=\sup\{\alpha>0\mid \exists C(\alpha)\,\,\mbox{s.t.} \int_Xe^{-\alpha\varphi}\omega^n <C(\alpha) \,\,\mbox{for all } \varphi\in P_G(X,\omega)\}.
$$
This is independent of the choice of $\omega$.
If $G$ is trivial, we denote it by just $\alpha(X)$. 
Then, he proved 
\begin{Fac}[Tian \cite{Tia87}]\label{fact:tian}
If $\alpha_G(X)>\frac{n}{n+1}$, then $X$ admits a K\"ahler-Einstein metric. 
\end{Fac}
Let us recall the following conjecture, which was finally formulated in \cite{Don02}. 

\begin{Conj}[cf.\ \cite{Yau90}, \cite{Tia97}, \cite{Don02}]\label{DTY}
Let $(X,L)$ be a smooth polarized variety. $X$ has a K\"{a}hler metric with constant scalar curvature $($cscK metric$)$ with K\"{a}hler class $c_{1}(L)$ if and only if $(X,L)$ is K-polystable.
In particular, if $X$ is a Fano manifold, then 
$X$ has a K\"ahler-Einstein metric if and only if $(X,\mathcal{O}_{X}(-K_{X}))$ is 
K-polystable. 
\end{Conj} 
\noindent
From the recent progress in Conjecture \ref{DTY} (in particular, \cite{Tia97}, \cite{Don05}, \cite{CT08}, \cite{Stp08}),  one direction is proved as follows. 
\begin{Fac}\label{fact:YTD}
Let $\mathrm{Aut}(X,L)$ be the group of holomorphic automorphisms of  a polarized manifold $(X,L)$.
If $\mathrm{Aut}(X,L)$ is discrete and $(X,L)$ admits cscK metrics, then it is K-stable.
\end{Fac}
The case where $\mathrm{Aut}(X,L)$ is not discrete is studied in \cite{Mab08} and \cite{Mab09}.
Combining Fact \ref{fact:tian} and \ref{fact:YTD}, we find that if $\alpha_G(X)>\frac{n}{n+1}$, then $(X,\mathcal{O}_{X}(-K_X))$ is K-polystable. 
The main theme in our paper is recovering this relation directly in algebro-geometric way. 
For that, we replace $\alpha_G(X)$ by the invariant in algebro-geometric context, which is often called  the (\textit{global}) \textit{log canonical threshold} defined by
\begin{equation}\label{eq:alpha}
	\lct_{G}(X)
	:=
	\inf_{m\in\mathbb{Z}_{>0}}
\inf
_{
\Sigma
} 
\mathrm{lct}\big(X,\frac{1}{m}\Sigma\big). 
\end{equation}
In the second infimum in (\ref{eq:alpha}), $\Sigma$ runs over the set of $G$-invariant  sublinear system of $\mid -mK_X\mid $. 
If $G$ is finite, then we can replace (\ref{eq:alpha}) by 
\begin{equation}\label{eq:alpha_finite} 
	\lct_{G}(X)
	=
	\inf_{m\in\mathbb{Z}_{>0}}
\inf_{
D
}
\mathrm{lct}\big(X,\frac{1}{m}D \big). 
\end{equation}
In particular, if $G$ is trivial, we denote it by just $\lct(X)$.
In the second infimum in (\ref{eq:alpha_finite}), $D$ runs over the set of $G$-invariant effective divisors which are linearly equivalent to $-mK_X$.
Let us recall that, in general, for an effective $\mathbb{Q}$-divisor $D$, the \textit{log canonical threshold} $\lct (X,D)$ is an invariant to measure the singularities of a pair $(X,D)$ as follows; 
\begin{equation}\label{eq:lct_D}
	\mathrm{lct}(X,D):=
	\sup \{c\in \mathbb{Q}_{>0} \mid (X,cD) \,\, \mbox{is log canonical}\}.
\end{equation}
In the appendix of \cite{CSD08} by Demailly, it is explained that $\lct_G(X)$ is equal to $\alpha_G(X)$ of Tian, for smooth $X$ with compact $G$. 
While $\alpha_G(X)$ is defined in differential geometric way and used for the existence problem of K\"ahler-Einstein metrics on Fano manifolds, $\lct_G(X)$ is defined and studied in purely algebro-geometric way.
So, we work with $\lct_G(X)$ instead of $\alpha_G(X)$. 

We work over an algebraically closed field $k$ with characteristic $0$, 
since we use the resolution of singularities for the equality (\ref{eq:lctD=lctDxA}). 
On the other hand, since that is the only point we need the assumption of characteristic,   
our main theorems \ref{Thm:main} and \ref{Thm:main_g} work up to dimension $3$ over an arbitrary algebraically closed field with positive characteristic as well. 

The main statement is as follows. 
\begin{Thm}\label{Thm:main}
Let $X$ be a $($log-canonical$)$ $\mathbb{Q}$-Fano variety with $\dim(X)=n$ and 
suppose that $\lct(X)> \frac{n}{n+1}$ $($resp.\ $\lct(X)\geq \frac{n}{n+1}$$)$. 
Then, $(X,\mathcal{O}_{X}(-K_X))$ is K-stable $($resp.\ K-semistable$)$. 
\end{Thm}
\noindent
We note that the log-canonicity of $X$ in the assumption naturally follows from the assumption that $\lct(X)\ge 0$. 
Furthermore, the first author proved in \cite{Od11} (modulo LMMP)
that K-semistability of a $\mathbb{Q}$-Fano variety $X$ implies the log-canonicity of $X$. 

We also note that this notion of K-stability 
implies that $X$ does not admit any non-trivial one parameter subgroup in $\Aut(X)$. 
Therefore, together with Matsushima's obstruction to K\"ahler-Einstein metrics \cite{Mat57}, we have
\begin{Cor}\label{aut:fin}
Let $X$ be a smooth Fano manifold over $\mathbb{C}$ with $\dim(X)=n$ and 
suppose that $\lct(X)>\frac{n}{n+1}$. Then, $\Aut(X)$ is finite. 
\end{Cor}
\noindent
Although K-stability in Theorem \ref{Thm:main} and the finiteness of $\mathrm{Aut}(X)$  in Corollary \ref{aut:fin} might seem to be stronger than Fact \ref{fact:YTD}, we can recover them in analytic way\footnote{This is pointed out to us by Professor Hiraku Nakajima.}.
In fact, by using Tian's estimate in \cite{Tia87}, we find that if $\alpha(X)$ (without $G$-action) is strictly greater than $n/(n+1)$, then the set of K\"ahler-Einstein metrics is compact. 
The set  of K\"ahler-Einstein metrics has a transitive action of the identity component $\Aut^{0}(X)$ of $\Aut(X)$ by Bando-Mabuchi \cite{BM85} and the connected component of its isotropy subgroup is a compact subgroup of $\Aut^{0}(X)$ (cf.\ \cite{Mat57}). 
Therefore, if $\Aut^{0}(X)$ is not trivial, the set of K\"ahler-Einstein metrics is non-compact, which is in contradiction with the condition on $\alpha(X)$. 
Hence,  $\mathrm{Aut}(X)$ is finite and K-polystability in Fact \ref{fact:YTD} is equivalent to K-stability. 

We remark that the first author (\cite{OdII}) also found similar proofs for the finiteness of $\mathrm{Aut}(X)$ by using K-stability in the case of general type varieties,
 which is well known, and of polarized Calabi-Yau varieties.
\begin{Ex}
{\rm (i)} For $\dim(X)=2$ case, it is easy to see that 
a blow up of general $n(\ge 5)$  points of the projective plane has the finite automorphism group $\Aut(X)$. 
On the other hand, it is known that  $\alpha(X)\ge \frac{2}{3}$ for $n\ge 6$ case and they have K\"ahler-Einstein metrics (see \cite{TY87} and \cite{Che08}). 
\\
{\rm (ii)} Let $X$ be a general smooth hypersurface of degree $n+1\ge 3$ in $\mathbb{P}^{n+1}$. 
Then, $\alpha(X)>\frac{n}{n+1}$ (cf.\ e.g.\ \cite{CPW07}). 
On the other hand, it is known that a smooth hypersurface has the finite automorphism  group, due to \cite{MM63}. 
 \end{Ex}

When we apply Fact \ref{fact:tian}, the group action of $G$ plays important role.
In fact, $\alpha(X)$ might not be large enough in general.
The large symmetry of $X$ by $G$ makes $\alpha_G(X)$ larger, i.e., $\alpha_{G'}(X)\ge \alpha_G(X)$ if $G\subset G'$. 
We remark that the compactness of $G$ is not necessarily assumed in our results. 

\begin{Ex}\label{ex:toric_MU}
{\rm (i)} For a symmetric toric Fano manifold, in the sense of Batyrev and Selivanova  \cite{BS99}, $\alpha_{G}(X)=1$ where $G$ is a non-connected compact subgroup of $\Aut(X)$ (whose identity component is the algebraic torus). 
However, we can see $\alpha(X)\le \frac{1}{2}$ due to \cite{Son03}.
\\
{\rm (ii)} 
Let $X$ be the so-called Mukai-Umemura $3$-fold.
 This is a compactification of the quotient $\mathrm{SL}(2,\mathbb{C})/\Gamma$ where $\Gamma$ is the icosahedral group.
Then, it is known that for an action of 
$G\cong \SO(3) (\subset \SL(2,\mathbb{C}))$, $\alpha_{G}(X)=\frac{5}{6}$ (cf.\ \cite{Don07}) but $\alpha(X)=\frac{1}{2}$ (cf.\ \cite{CS08}). 
\end{Ex}
\noindent
Therefore, it is important to establish the $G$-equivariant version of Theorem \ref{Thm:main}. 
We have the following partial results in this direction. 
First, the next proposition follows straightforwardly if we apply
the Borel fixed point theorem (cf.\ 
\cite[Chapter 4, Theorem 6.6]{Mil05}
) to the natural action of $G$ on $\mid -mK_X \mid$ for $m\in \mathbb{Z}_{>0}$  and take into account the lower semicontinuity of log-canonical threshold with respect to a variation of divisors (cf.\ \cite[Example 9.5.41]{Laz04}). 

\begin{Prop}\label{solvable.alpha.same}
For any $\mathbb{Q}$-Fano variety, 
$\lct_{G}(X)=\lct(X)$ if $G$ is a connected and solvable algebraic group. 
\end{Prop}
\noindent
Then, we have
\begin{Cor}\label{cor:solvable}
Let $X$ be a $($log-canonical$)$ $\mathbb{Q}$-Fano variety with $\dim(X)=n$ and 
suppose that $\lct_{G}(X)> \frac{n}{n+1}$  $($resp. $\lct_{G}(X) \ge \frac{n}{n+1}$$)$ with some connected solvable algebraic subgroup $G\subset \Aut(X)$.
Then, $(X,\mathcal{O}_{X}(-K_X))$ is K-stable $($resp. K-semistable$)$.
Furthermore, if $X$ is smooth, $\lct_{G}(X)>\frac{n}{n+1}$ implies that $G$ is trivial. 
\end{Cor}
\noindent
We note that the triviality of $G$ follows from Corollary \ref{aut:fin}. 

Also, we have the following in completely similar manner as the proof of Theorem \ref{Thm:main}.
\begin{Thm}\label{Thm:main_g}
Let $X$ be a $($log-canonical$)$ $\mathbb{Q}$-Fano variety with $\dim(X)=n$ 
and $G$ be a $($not necessarily compact$)$ subgroup of $\mathrm{Aut}(X)$.
Suppose that $\lct_{G}(X)> \frac{n}{n+1}$  $($resp. $\lct_{G}(X) \ge \frac{n}{n+1}$$)$. 
Then, $(X,\mathcal{O}_{X}(-K_X))$ is $G$-equivariantly K-stable $($resp. K-semistable$)$. 
\end{Thm}
\noindent
Here, we introduced new notions of \textit{$G$-equivariant K-stability} (
resp.\ \textit{$G$-equivariant K-semistability}), which are a priori weaker than the original notions of K-stability (resp.\ K-semistability)
\footnote{We expect that actually they are equivalent to the original, though we do not know the proof.}.
Their definitions will be explained in Section \ref{subsec:Kstab}.
Then, we have the following corollary thanks to the theorem of Matsushima \cite{Mat57} again. 
\begin{Cor}\label{aut:ss}
Let $X$ be a smooth Fano manifold over $\mathbb{C}$ with $\dim(X)=n$ and 
suppose that $\lct_{G}(X) > \frac{n}{n+1}$ with some connected compact subgroup $G\subset \Aut(X)$. 
Then, $\Aut(X)$ is semisimple. 
\end{Cor}
\noindent
We also have analytic proof of Corollary \ref{aut:ss}, as well as Corollary \ref{aut:fin}, which is explained in the last section.  
\begin{Ex}
In Example \ref{ex:toric_MU} (ii), $\Aut(X)$ is isomorphic to $\PGL(2, \mathbb{C})$, which is semisimple. 
\end{Ex}
\begin{Rem}
In Example \ref{ex:toric_MU} (i), $\Aut(X)$ is \textit{not} semisimple, although $\alpha_G(X)=1$.
In fact, $G$ is \textit{not} connected.  
Therefore, the connectedness assumption of $G$ is necessary in Corollary \ref{aut:ss}. 
\end{Rem}
We have two keys to the algebro-geometric proof of Theorem \ref{Thm:main} and \ref{Thm:main_g};
one is a relation between the log canonical thresholds and the Seshadri constants, and the other is an estimate of the Donaldson-Futaki invariants.
The Seshadri constant is also a key in \cite{HKLP10}.
They used bend-and-break techniques and their related consequences, to
yields the necessary estimates of the Seshadri constants. 
The estimate of the Donaldson-Futaki invariants is an application of the first author's formula \cite{Od09} to compute them.

This paper is organized as follows.
In Section \ref{sec:Preliminary}, we recall the definitions of terminologies and facts needed for the proof. 
In Section \ref{lct.Sesh}, we prove the first step.
In Section \ref{formula}, we prove the second step.
In Section \ref{proof}, we integrate the materials to complete the proof of theorems and corollaries.

\vspace{0.3cm}
\noindent
\textbf{Acknowledgements.}
Y.O is supported by the Grant-in-Aid for Scientific Research (KAKENHI No.\ 21-3748) and the Grant-in-Aid for JSPS fellows. Y.S is supported by MEXT, Grant-in-Aid for Young Scientists (B), No.\ 22740041 and he is also a member of the Global COE program ``Education and Research Hub for Mathematics-for-Industry" in Kyushu University. 

This work was started roughly in the conference of ``Younger Generation in Complex and Algebraic Geometry (joint conference with ``Complex Analysis  in Kumamoto")" in Kumamoto, Japan. We thank the organizers for inviting us. 
This work is done partially during the second author's visit to Postech, Korea, in August 2010. He thanks Professor Jihun Park for his kind hospitality and useful comments on this work. 
We also thank Professor Hiraku Nakajima for pointing out the analytic proof of the finiteness of $\mathrm{Aut}(X)$ (Corollary \ref{aut:fin}) to us.

\section{Preliminary}\label{sec:Preliminary}
In this section, we make clear the definitions of the terminologies in the introduction.
We call $X$ a $\mathbb{Q}$-Fano variety if  $-K_X$ is an ample $\mathbb{Q}$-Cartier $\mathbb{Q}$-divisor.

\subsection{The log-canonicity and the log canonical thresholds}
Consult \cite{Kol95} and the textbook \cite[Section 9]{Laz04} for the details. 
Let $(X,D)$ be a pair of a normal variety $X$ and an effective $\mathbb{Q}$-divisor $D$.
Throughout this subsection, we assume that $K_X$ is $\mathbb{Q}$-Cartier.
Let $\pi:X'\to X$ be a log resolution of $D$, i.e., $\pi$ is a proper birational morphism such that $X'$ is smooth and the divisor $\pi^*D+E$ has a simple normal crossing support, where $E$ is the exceptional divisor of $\pi$.
Let $K_{X'/X}:=K_{X'}-\pi^*K_{X}$. 
Then, we denote 
$$
%\begin{equation}\label{eq:discrepancy}
	K_{X'/X}-\pi^*D=\sum a_{i}E_{i}, 
%\end{equation} 
$$
where $a_{i} \in \mathbb{Q}$ and $E_{i}$ runs over the set of divisors of $X'$ supported on the exceptional locus or the support $\Supp(\pi^{-1}_{*}D)$ of the strict transform of $D$. 
The pair $(X,D)$ is called \textit{log canonical} 
if and only if $a_{i}\ge -1$ for any $E_{i}$. This notion is independent of 
the choice of log resolution. 
From the definition (\ref{eq:lct_D}), the log canonical threshold is 
determined as 
$$
	\lct(X,D)=\min_{E_{i}\subset X'}
	\bigg\{ 
		\frac{1+\mathrm{ord}_{E_{i}}(K_{X'/X})}{\mathrm{ord}_{E_{i}}(D)} 
	\bigg\},
$$
where $\mathrm{ord}_{E_{i}}(K_{X'/X})=a_{i}$ and $\mathrm{ord}_{E_{i}}(D)$ is the coefficient of $E_{i}$ in $\pi^{*}D$. 
The log canonical threshold is also independent of the choice of log resolution. 
More generally, if $\pi'$ is a proper birational morphism (possibly not a log resolution),  the fact that such $\pi' \colon X' \rightarrow X$ is dominated by a log resolution implies 
\begin{equation}\label{eq:lct_inequality} 
	\lct(X,D) \le \min_{E'_{i}\subset X'}
	\bigg\{
		\frac{1+\mathrm{ord}_{E'_{i}}(K_{X'/X})}{\mathrm{ord}_{E'_{i}}(D)}
	\bigg\}
\end{equation}
where $E'_{i}$ runs over the set of divisors of $X'$ supported on the exceptional locus 
or the support  $\Supp(\pi'^{-1}_{*}D)$ of the strict transform of $D$. 
This is one of the essential observations in the first step of the proof.

The log canonical thresholds can be defined similarly for linear systems and ideals by using their log resolutions (cf. \cite{Laz04}) as follows. 
Let $L$ be an ample line bundle on $X$.
Let $\Sigma$ be a sublinear system of $|L|$.
We say that a proper birational morphism $\pi:X'\to X$ is a \textit{log resolution of} $\Sigma$ if $X'$ is smooth and there exist an effective divisor $F$ on $X'$ and a linear system $\Sigma' \subset |\pi^*L-F|$ such that
$$
	\pi^* \Sigma =F+\Sigma' ,
$$
$F+E$ has a simple normal crossing support and $\Sigma'$ is base point free, where $E$ is the exceptional divisor of $\pi$.
Then, we denote
$$
	K_{X'/X}-cF=\sum a_{i}E_{i}+D
$$
where $E_{i}$ are exceptional divisors of $\pi$ and $D$ is non-exceptional parts. 
We say that a pair $(X,\Sigma)$ is \textit{log canonical} if $a_{i}\ge -1$ for any $E_i$.
Then, we can define the \textit{log canonical threshold}  $\lct(X,\Sigma)$ by
\begin{equation}\label{eq:alg_lct_LS}
\lct(X,\Sigma):=\sup\{c\mid K_{X'/X}-cF=\sum a_{i}E_{i}+D \mbox{ with } a_{i}\geq -1\}. 
\end{equation}
Here, $E_{i}$ and $D$ are as above. 
We note that the definition of $\lct(X, \Sigma)$ in the appendix \cite{CSD08} uses the complex singularity exponent, but it is equivalent to (\ref{eq:alg_lct_LS}).
The equivalence follows from a standard argument for the correspondence of the complex singularity exponent and the log canonical threshold for divisors.
We note that the log canonical threshold $\lct(X,\Sigma)$ coincides with $\lct(X,D)$ for some  effective $\mathbb{Q}$-divisor $D$, which is $\mathbb{Q}$-linearly equivalent 
to a member of $\Sigma$, by \cite[Proposition 9.2.26]{Laz04}. Furthermore, $\lct(X,\Sigma)$ also coincides with the log canonical threshold for a coherent ideal sheaf $\lct(X,I)$, where $I$ is the base ideal sheaf of $\Sigma$ by \cite[Example 9.2.23]{Laz04}. 
Let $I\subset \mathcal{O}_X$ be a non-zero ideal of $X$.
We say that $\pi:X'\to X$ as before is \textit{a log resolution} of $I$ if $X'$ is smooth and there exists an effective divisor $F$ on $X'$ such that
$$
	\pi^{-1}I=\mathcal{O}_{X'}(-F),
$$
$F+E$ has a simple normal crossing support.
Then, we can define $\lct(X,I)$ as before as 
$$
\lct(X,I):=\sup\{c\mid K_{X'/X}-cF=\sum a_{i}E_{i}+D \mbox{ with } a_{i}\geq -1\}. 
$$ 
Here, $E_{i}$ are exceptional divisors of $\pi$ and $D$ is a non-exceptional part, again.

\subsection{K-stability}\label{subsec:Kstab}
Consult \cite[Chapter 2, especially 2.3]{Don02}, \cite[Section 3]{RT07} or \cite[section  2]{Od09} for more details. 
Let $(X,L)$ be an $n$-dimensional polarized variety. 
A \textit{test configuration} (resp.\ a \textit{semi test configuration}) for $(X,L)$ is a polarize scheme $(\mathcal{X},\mathcal{L})$ with a $\mathbb{G}_m$-action on $(\mathcal{X},\mathcal{L})$ and a proper flat morphism $\Pi\colon \mathcal{X}\to\mathbb{A}^1$ such that (i) $\Pi$ is $\mathbb{G}_m$-equivariant for the multiplicative action of $\mathbb{G}_m$ on $\mathbb{A}^1$, (ii) $\mathcal{L}$ is relatively ample (resp. relatively semi-ample), and (iii) $(\mathcal{X},\mathcal{L})\mid_{\Pi^{-1}(\mathbb{A}^1-\{0\})}$ is $\mathbb{G}_m$-equivariantly isomorphic to  $(X,L^{\otimes r})\times (\mathbb{A}^1-\{0\})$ for some positive integer $r$. 
If $\mathcal{X}\simeq X\times \mathbb{A}^1$, we call $(\mathcal{X},\mathcal{L})$ a \textit{product test configuration}. 
Moreover, if $\mathbb{G}_m$ acts trivially, we call it a trivial test configuration.
A test configuration $(\mathcal{X}, \mathcal{L})$ is said to be \textit{almost trivial}
if $\mathcal{X}$ is $\mathbb{G}_m$-equivariantly isomorphic to a trivial test configuration away from a closed subscheme of codimension at least $2$ (cf.\ 
\cite[Definition 3.3]{Od12}, \cite[Definition 1]{Stp11}). 

Let $P(k):=\dim H^0(X,L^{\otimes k})$, which is a polynomial in $k$ of degree $n$ due to Riemann-Roch theorem. 
Since the $\mathbb{G}_m$-action preserves the central fibre $\mathcal{X}_0$ of $\mathcal{X}$, $\mathbb{G}_m$ acts also on $H^0(\mathcal{X}_0,\mathcal{L}^{\otimes K}\mid_{\mathcal{X}_0})$, where $K\in \mathbb{Z}_{>0}$. 
Let $w(Kr)$ be the weight of the induced action on the highest exterior power of $H^0(\mathcal{X}_0,\mathcal{L}^{\otimes K}\mid_{\mathcal{X}_0})$, which is a polynomial of $K$ of degree $n+1$ due to the Mumford's droll Lemma (cf.\ \cite[Lemma 2.14]{Mum77} and \cite[Lemma 3.3]{Od09}) and Riemann-Roch theorem.
Here, the \textit{total weight} of an action of $\mathbb{G}_{m}$ on some finite-dimensional vector space is defined as the sum of all weights, where the \textit{weights} mean the exponents of eigenvalues which should be powers of $t\in \mathbb{A}^1$.
Let us take $rP(r)$-th power and 
SL-normalize the action of $\mathbb{G}_{m}$ on 
$(\Pi_{*}\mathcal{L})|_{\{0\}}$, then the corresponding normalized weight on 
$(\Pi_{*}\mathcal{L}^{\otimes K})|_{\{0\}}$ 
is $\tilde{w}_{r,Kr}:=w(k)rP(r)-w(r)kP(k)$, where $k:=Kr$. It is a 
polynomial of form 
$\sum_{i=0}^{n+1}e_{i}(r)k^{i}$ of degree $n+1$ in $k$ for $k \gg 0$, with coefficients which are also 
polynomial 
of degree $n+1$ in $r$ 
for $r \gg 0$ : $e_{i}(r)=\sum_{j=0}^{n+1}e_{i,j}r^{j}$ for $r \gg 0$. 
Since the weight is normalized, $e_{n+1,n+1}=0$. 
The coefficient $e_{n+1,n}$ is 
called the \textit{Donaldson-Futaki invariant} of the test configuration, which we denote by $\DF(\mathcal{X},\mathcal{L})$. 
For an arbitrary \textit{semi} test configuration $(\mathcal{X},\mathcal{L})$ 
of order $r$, we can also define the Donaldson-Futaki invariant as well by setting $w(Kr)$ as the total weight of the  induced action on $H^{0}(\mathcal{X}, \mathcal{L}^{\otimes K})/tH^{0}(\mathcal{X}, \mathcal{L}^{\otimes K})$ (cf. \cite{RT07}).
We say that $(X,L)$ is K-semistable if and only if $\DF\ge 0$ for any non-trivial test configuration.
We say that $(X,L)$ is K-stable if and only if $\DF>0$ for all test configuration which are not almost trivial. 
We also say that $(X,L)$ is K-polystable if and only if $\DF\ge 0$ for all test configuration which are not almost trivial, and $\DF=0$ only if a test configuration is isomorphic to a product test configuration away from a closed subscheme of codimension at least $2$. 
\footnote{K-stability and K-polystability in this paper are slightly weaker than the original in \cite{Don02} to  avoid the pathological test configurations found recently by Li-Xu \cite[Example 1]{LX11}. }

Now, we define $G$-equivariant K-stability (resp.\ $G$-equivariant K-semistability) as follows.
We say that a test configuration $(\mathcal{X,L})$ of a polarized variety $(X,L)$ is \textit{$G$-equivariant} if it is equipped with an extension of the natural $G$-action on $(\mathcal{X},\mathcal{L})|_{\Pi^{-1}(\mathbb{A}^{1}-\{0\})}$ (which fixes coordinates of $\mathbb{A}^{1}$) to the whole space $(\mathcal{X},\mathcal{L})$. 
We note that the action of $G$ naturally commutes with the $\mathbb{G}_{m}$-action.
Then, $G$-equivariant K-stability (resp.\ $G$-equivariant K-semistability) in Theorem \ref{Thm:main_g} means that the Donaldson-Futaki invariant of an arbitrary $G$-equivariant test configuration $(\mathcal{X},\mathcal{L})$, which is not almost trivial (resp. trivial),  is positive (resp.\ non-negative). 
Therefore, $G$-equivariant K-stability of $(X,L)$ implies that $\Aut(X,L)$ does not include any algebraic subgroup which is isomorphic to  $\mathbb{G}_{m}$ and commutes with $G$.

We end this subsection with a small remark on an extension of the framework above. 
If we take a test configuration (resp.\ semi test configuration) $(\mathcal{X}, \mathcal{L})$, 
we can think of a new test configuration (resp.\ semi test configuration)  $(\mathcal{X}, \mathcal{L}^{\otimes a})$ with $a \in \mathbb{Z}_{>0}$. 
From the definition of Donaldson-Futaki invariant above, we easily see that 
$\DF((\mathcal{X},\mathcal{L}^{\otimes a}))=a^{n}\DF((\mathcal{X},\mathcal{L}))$. 
Therefore, we can define  K-stability (also K-polystability and K-semistability) of a pair $(X,L)$ of a projective scheme $X$ and an ample \textbf{$\mathbb{Q}$-line} bundle $L$. 

\subsection{Seshadri constants}
Let $J\subset \mathcal{O}_X$ be a coherent ideal on $X$.
The Seshadri constant of $J$ with respect to an ample $\mathbb{Q}$-line bundle $L$ 
is defined by 
$$
	\mathrm{Sesh}(J;(X, L)):=\sup \{c>0\mid \pi^*L(-cE)\,\,\mbox{is ample}\},
$$
where $\pi\colon X'\to X$ is the blow up of $X$ along $J$.

\subsection{Flag ideals}\label{subsec:flag}
See \cite[Section 3]{Od09} and \cite[Section 3]{RT07} for the details.
We say that  a coherent ideal $\mathcal{J}\subset \mathcal{O}_{X\times \mathbb{A}^{1}}$  is a \textit{flag ideal} if it is of the form 
\[
\mathcal{J}=I_{0}+I_{1}t+I_{2}t^{2}+\cdots +I_{N-1}t^{N-1}+(t^N), 
\]
where $I_{0} \subset I_{1} \subset \cdots I_{N-1} \subset \mathcal{O}_{X}$ is a sequence of coherent ideals of $X$.
By using a flag ideal, we construct a special class of semi test configurations as follows.
For a flag ideal $\mathcal{J}$, let $(\mathcal{B}:=Bl_{\mathcal{J}}(X\times \mathbb{A}^{1}),\mathcal{L}(-E))$ be the blow up $\Pi$ of $(X\times \mathbb{A}^{1})$ along $\mathcal{J}$, where $\mathcal{O}(-E)=\Pi^{-1}\mathcal{J}$, $\mathcal{L}:=\Pi^*p_1^*L^{\otimes r}$ 
with $r \in \mathbb{Z}_{>0}$ and $p_{1}:X\times \mathbb{A}^{1}\to X$.
We assume that $\mathcal{L}(-E)$ is semiample. 
Then, $(\mathcal{B},\mathcal{L}(-E))$ with the induced action from the usual action of $\mathbb{G}_m$ on $X\times \mathbb{A}^1$ (i.e, $\mathbb{G}_m$ acts only the second factor) defines a semi test configuration.
Remark that if $\mathcal{J}=(t^N)$, then $(\mathcal{B},\mathcal{L}(-E))$ defines 
{\color{black}{a trivial}} test configuration, because the blow up morphism $\mathcal{B}\to X\times \mathbb{A}^1$ is trivial.
The following says that it suffices to consider all semi test configurations only type of $(\mathcal{B},\mathcal{L}(-E))$ in order to show K-stability. 
Note that the following proposition is a corrected version of \cite[Corollary 3.11 (ii)]{Od09} in \cite{Od12}. 

\begin{Prop}[{cf.\ \cite[Corollary 3.11 (ii)]{Od09}, \cite[Corollary 3.6]{Od12}}] \label{b.e}
Suppose that $X$ is normal. 
$(X,L)$ is K-stable if and only if $DF(\mathcal{B},\mathcal{L}(-E))>0$ for all flag ideals which are not of the form $(t^N)$, and $r \in \mathbb{Z}_{>0}$ such that $\mathcal{B}$ is 
normal and $\mathcal{L}(-E)$ is semi-ample.
\end{Prop}
Remark that the normality of $\mathcal{B}$ can be assumed without loss of generality.
In fact, by normalizing, $(\mathcal{B}, \mathcal{L}(-E))$ can be made a test configuration with respect to some (possibly different) flag ideal with smaller Donaldson-Futaki invariant.

\section{The log canonical thresholds and Seshadri constants}\label{lct.Sesh}
We prove the first step of the proof. 
Let $X$ be a $\mathbb{Q}$-Fano variety. 
Let $\mathcal{J}$ be a flag ideal.
Let $\Pi:\mathcal{B}\to X\times \mathbb{A}^1$ be the blow up of $X\times \mathbb{A}^1$ along $\mathcal{J}$.
Assume that $\mathcal{B}$ is normal.
We denote
\begin{eqnarray*}
		K_{\mathcal{B}/X\times \mathbb{A}^1}
	&=&
		\sum_i a_iE_i,
	\\
		\Pi^{*}(X\times \{0\})
	&=&
		\Pi^{-1}_*(X\times \{0\})+\sum_i b_i E_i, 
	\\
		\Pi^{-1}\mathcal{J}
	&=&
		\mathcal{O}_{\mathcal{B}}(-\sum_i c_i E_i).
\end{eqnarray*}
Remark that these three divisors is supported only in the central fibre of $X\times \mathbb{A}^1$.
Then, the first step of the proof is as follows.
\begin{Prop}\label{prop:first}
Let $X$ be a $\mathbb{Q}$-Fano variety.
If $\lct(X)>0$, then we have
$$
	\mathrm{Sesh}(\mathcal{J}, (X\times \mathbb{A}^1, \mathcal{O}_{X\times \mathbb{A}^{1}}(-K_{X\times \mathbb{A}^1})))
	\le
	\frac{1}{\lct(X)}\min_i \bigg\{ \frac{a_i-b_i+1}{c_i}\bigg\}.
$$
\end{Prop}
\begin{proof}
Since
$$
	\mathrm{Sesh}(\mathcal{J}, (X\times \mathbb{A}^1, \mathcal{O}_{X\times \mathbb{A}^1}(-K_{X\times \mathbb{A}^1})))
	=
	\min_j \big\{ \mathrm{Sesh}(I_j, (X,\mathcal{O}_{X}(-K_X))) \big\}
$$
(cf. \cite[Corollary 5.8]{RT07}), it suffices to show that
$$
	\mathrm{Sesh}(I_0, (X,\mathcal{O}_{X}(-K_X)))
	\le
	\frac{1}{\lct(X)}\min_i 
	\bigg\{
		 \frac{a_i-b_i+1}{c_i}
	\bigg\}.
$$
Take $c\in \mathbb{Q}_{>0}$ so that $c<\mathrm{Sesh}(I_0, (X,\mathcal{O}_{X}(-K_X)))$.
Let $E$ be the exceptional divisor of the blow up $\pi:X'\to X$ along $I_0$.
Since $\pi^*(-K_X)-cE$ is ample, 
$H^0(X,I_{0}^{mc}\mathcal{O}_{X}(-mK_{X}))$ has positive dimension 
for sufficiently divisible positive integer $m$ and we can take a linear system 
$\Sigma$ which corresponds to that space $H^0(X,I_{0}^{mc}\mathcal{O}_{X}(-mK_{X}))$ as a subspace of $H^0(X,\mathcal{O}_{X}(-mK_{X}))$. Let us take an effective  $\mathbb{Q}$-divisor $D$ as $mD \in \Sigma$.

From now on, we work with $X\times \mathbb{A}^1$ instead of $X$.
For a pair $(Y,\Delta)$ of a normal algebraic variety $Y$ and an effective $\mathbb{Q}$-divisor $\Delta$ on $Y$, and an effective $\mathbb{Q}$-divisor $F$ on $Y$, we denote the log canonical threshold of $((Y, \Delta);F)$ 
$$
	\sup\{c\in \mathbb{Q}_{>0}\mid 
	(Y, (\Delta + cF))\,\,
	\mbox{is log canonical} \}
$$
by $\lct((Y,\Delta); F)$.
Then, we get
\begin{eqnarray}
	\nonumber
		\lct(X)
	&\le&
		\lct(X,D)
	\\
	\label{eq:lctD=lctDxA}
	&=&
		\lct((X\times \mathbb{A}^1, X\times \{0\}); 
		I_{D\times \mathbb{A}^1}
		)
	\\
	\label{eq:DA_cI0}
	&\le&
		\lct((X\times \mathbb{A}^1, X\times \{0\}); cI_0)
	\\
	\nonumber
	&\le&
		\lct((X\times \mathbb{A}^1, X\times \{0\}); c\mathcal{J})
	\\
	\nonumber
	&=&
		\frac{1}{c}\lct((X\times \mathbb{A}^1, X\times \{0\}); \mathcal{J})
	\\
	\label{eq:estimate_coefficients}
	&\le&
		\frac{1}{c}\min_i \bigg\{
			\frac{a_i-b_i+1}{c_i}
		\bigg\}.
\end{eqnarray}
The equality (\ref{eq:lctD=lctDxA}) follows from the inversion of adjunction of the log canonicity, which can be seen easily by taking the log resolution formed of $\widetilde{\pi}:X'\times \mathbb{A}^1\to X\times \mathbb{A}^1$ where $\pi:X'\to X$ is a log resolution of $(X,D)$ for this case. 
The inequality (\ref{eq:DA_cI0}) follows by taking a log resolution of the blow up $Bl_{I_0}(X\times \mathbb{A}^1)$ of $X\times \mathbb{A}^1$ along the ideal $I_0\subset \mathcal{O}_{X\times \mathbb{A}^1}$.
The last inequality (\ref{eq:estimate_coefficients}) follows from the inequality (\ref{eq:lct_inequality}).
In fact, 
$$
	\lct((X\times \mathbb{A}^1, X\times \{0\}); \mathcal{J})
	\le
	\min_i 
	\bigg\{
	\frac{1+\mathrm{ord}_{E_i}(K_{\mathcal{B}/X\times\mathbb{A}^1})-\mathrm{ord}_{E_i}(X\times \{0\})}{\mathrm{ord}_{E_i}(\mathcal{J})}
	\bigg\}. 
$$
Therefore, 
$$
	c\le  \frac{1}{\lct(X)}
	\min_i \bigg\{
		\frac{a_i-b_i+1}{c_i}
	\bigg\}
$$
for any $c<\mathrm{Sesh(I_0, (X,\mathcal{O}_{X}(-K_X)))}$.
The proof is completed. 
\end{proof}

\section{Estimates of the Donaldson-Futaki invariants}\label{formula}

We prove the second step of the proof of Theorem \ref{Thm:main} in this section. 
This is an application of the following formula in \cite{Od09} to compute the Donaldson-Futaki invariant for a semi test configuration $(\mathcal{B}, \mathcal{L}(-E))$ derived from a flag ideal $\mathcal{J}\subset \mathcal{O}_{X\times \mathbb{A}^1}$. 

\subsection{The formula for the Donaldson-Futaki invariants and its decomposition} 

Let us start from recalling the formula from \cite{Od09}. 

\begin{Thm}[{\cite[Theorem 3.2]{Od09}}]\label{DF.formula}
Let $X, L, \mathcal{J}, \mathcal{B}, \mathcal{L}$ and $E$ as before 
$($cf.\ Subsection \ref{subsec:flag}$)$.
Let $(\overline{\mathcal{B}}:=Bl_{\mathcal{J}}(X\times \mathbb{P}^{1}),\overline{\mathcal{L}}(-E))$ be its natural compactification, i.e., $\bar{\mathcal{L}}:=\Pi^*p_1^*L$ (extension of $\mathcal{L}$) 
where $\Pi\colon Bl_{\mathcal{J}}(X\times \mathbb{P}^1)\rightarrow X\times\mathbb{P}^1$ is the blow up morphism and $p_1$ is the projection morphism. 
Suppose that $\mathcal{L}(-E)$ on $\mathcal{B}$ is semi-ample.
Then, if $\mathcal{B}$ is normal,
we have
\begin{eqnarray}
	\nonumber
	&&
		2(n!)((n+1)!)\DF (\mathcal{B},\mathcal{L}(-E)) 
	\\
	\nonumber
	&\quad&
  		=-n(L^{n-1}.K_{X})(\overline{\mathcal{L}}(-E)^{n+1})
		+(n+1)(L^{n})
		(\overline{\mathcal{L}}(-E)^{n}.K_{\bar{\mathcal{B}}/\mathbb{P}^{1}}) 
	\\
	\nonumber
	&\quad&
		=
		-n(L^{n-1}.K_{X})(\overline{\mathcal{L}}(-E)^{n+1})+(n+1)(L^{n})
		(\overline{\mathcal{L}}(-E)^{n}.\Pi^{*}(p_{1}^{*}K_{X})) 
	\\
	\nonumber
	&&
		\quad
		+(n+1)
		(L^{n})(\overline{\mathcal{L}}(-E)^{n}.K_{\bar{\mathcal{B}}/X\times \mathbb{P}^{1}}). 
\end{eqnarray}
In the above, the intersection numbers $(L^{n-1}.K_X)$ and $(L^n)$ 
are taken on $X$. On the other hand, $K_{\bar{\mathcal{B}}/X\times \mathbb{P}^1}
:=K_{\bar{\mathcal{B}}}-\Pi^*K_{X\times \mathbb{P}^1}$ is an exceptional divisor 
on $\bar{\mathcal{B}}$ and thus 
$((\bar{\mathcal{L}})(-E))^{n}.\Pi^{*}(p_{1}^{*}K_{X}))$ 
and $((\bar{\mathcal{L}})(-E))^{n}.K_{\bar{\mathcal{B}}/X\times \mathbb{P}^{1}})$ 
are intersection numbers taken on $\bar{\mathcal{B}}$. 
\end{Thm}
Now, we apply Theorem \ref{DF.formula} to Fano case which is our concern. 
Let $\overline{\mathcal{L}}=\Pi^*p_{1}^*(\mathcal{O}_{X}(-rK_X))$ where $r\in \mathbb{Z}_{>0}$ such that $\mathcal{L}(-E)$ on $\mathcal{B}$ is semi-ample. 
In particular, we have 
\begin{equation}\label{ses}
\frac{1}{r}\le \mathrm{Sesh}(\mathcal{J}, (X,\mathcal{O}_{X\times \mathbb{A}^{1}}(-K_{X\times \mathbb{A}^1}))).
\end{equation}

\noindent
On the other hand, Theorem \ref{DF.formula} implies
\begin{eqnarray}\label{10}
	\nonumber
	&&
		2(n!)((n+1)!)\DF(\mathcal{B}, \mathcal{L}(-E))
	\\
	\nonumber
	&&\quad
		=
		-\big((\overline{\mathcal{L}}-E)^n.\overline{\mathcal{L}}+nE\big)
		+
		(n+1)r
		\big((\overline{\mathcal{L}}-E)^n.K_{\mathcal{B}/X\times \mathbb{A}^1}\big)
	\\
	\label{eq:decomposition}
	&&\quad
		=
		-\big((\overline{\mathcal{L}}-E)^n.\overline{\mathcal{L}}\big)
		+
		\big((\overline{\mathcal{L}}-E)^n.
		((n+1)rK_{\mathcal{B}/X\times \mathbb{A}^1}-nE)\big).
\end{eqnarray}
We estimate the first and the second terms in (\ref{eq:decomposition}) separately. 
For the estimation of the second term, we use the bound for Seshadri constant  (\ref{ses}).

\subsection{Estimation of the first term}

Let us start from estimating the first term. 
Let us denote $\dim \mathrm{Supp}(\mathcal{O}_{X\times \mathbb{A}^1}/\mathcal{J})$ 
by $s$. In our estimation, we will use the following elementary decomposition of polynomial. 
\begin{Lem}\label{Lem:positivity}
There exist positive constants $\gamma_i$ and positive constants $\delta_{i,j}$ $($$0\le i \le n-1,\, 1\le j\le n-1$$)$ with $0<\delta_{i,j}<1$ such that the following equality of polynomials holds. 
\begin{eqnarray}
	\nonumber
	&&
		S^{n-1}+S^{n-2}(S-T)+\cdots+(S-T)^{n-1}	
	\\
	\label{eq:polynomial_rewritten}
	&&\quad
		=
		\sum_{i=0}^{n-1}
		\gamma_{i}
		(S-\delta_{i,1}T)\cdots
		(S-\delta_{i,n-1}T). 
\end{eqnarray}
\end{Lem}
\begin{proof}[proof of Lemma \ref{Lem:positivity}]
Let us put
$$
	f(S,T):=S^{n-1}+S^{n-2}(S-T)+\cdots+(S-T)^{n-1}.
$$
It is an elementary fact that, for generic $\{\delta_{i,j}\in \mathbb{R}_{>0}\}_{0\le i \le n-1,\, 1\le j\le n-1}$, 
$$
	\big\{
		g(S,T,\{\delta_{i,j}\}_j)
	\big\}_{0\le i \le n-1}
	:=
	\bigg\{
		\prod_{j=1}^{n-1} (S-\delta_{i,j}T)
	\bigg\}_{0\le i \le n-1}
$$ 
constitutes a basis of the vector space of homogeneous polynomials in $S,T$ of degree $n-1$.
Hence, for generic $\{\delta_{i,j}\}_{i,j}$, $f$ can be written as a linear combination of $g(S,T,\{\delta_{i,j}\}_j)$, i.e., there exist constants $\gamma_i$ such that
$$
	f(S,T)=\sum_i \gamma_i g(S,T,\{\delta_{i,j}\}_j).
$$
In particular, $\gamma_i=1$ for all $i$ when 
\begin{equation}\label{eq:coefficients}
	\delta_{i,j}
	=
	\left\{\begin{array}{cc}
	0 & \mbox{if } i+1\le j
	\\
	1 & \mbox{otherwise}.
	\end{array}\right.
\end{equation}
Perturbing $\{\delta_{i,j}\}_{i,j}$ in (\ref{eq:coefficients}), we get $\gamma_i$ and $\{\delta_{i,j}\}_{i,j}$ satisfying (\ref{eq:polynomial_rewritten}). 
Here, we complete the proof of Lemma \ref{Lem:positivity}. 
\end{proof}

\noindent
Think of the equality (\ref{eq:polynomial_rewritten}) substituted $S$ by $\mathcal{L}$ and $T$ by $E$.
Note that $\overline{\mathcal{L}}^{n+1}=0$. 
Hence, if $s>0$ and $\mathcal{J}$ is not of the form $(t^N)$, Lemma \ref{Lem:positivity} implies 
\begin{eqnarray*}
		-\big((\overline{\mathcal{L}}-E)^n.\overline{\mathcal{L}}\big)
	&=&
		\overline{\mathcal{L}}^{n+1}
		-(\overline{\mathcal{L}}-E)^{n+1}
		-E.(\overline{\mathcal{L}}-E)^n
	\\
	&=&
		E.\overline{\mathcal{L}}.
		\{
			\overline{\mathcal{L}}^{n-1}+
			\overline{\mathcal{L}}^{n-2}.(\overline{\mathcal{L}}-E)
			+\cdots+
			(\overline{\mathcal{L}}-E)^{n-1}	
		\}
	\\
	&=&
		E.\overline{\mathcal{L}}.
		\bigg(\sum_{i=0}^{n-1}
		\gamma_{i}
		(\overline{\mathcal{L}}-\delta_{i,1}E).\cdots.
		(\overline{\mathcal{L}}-\delta_{i,n-1}E)\bigg)
	\\
	&>&
		0.
\end{eqnarray*}
The last inequality follows from  that $E.\overline{\mathcal{L}}$ is a non-zero effective cycle.
If $s=0$, then it easily follows that
$$
	-\big((\overline{\mathcal{L}}-E)^n.\overline{\mathcal{L}}\big)=0.
$$
Summing up, we proved the following on the first term of (\ref{eq:decomposition}). 
\begin{Prop}\label{prop:positivity_first_term}
	$-\big((\overline{\mathcal{L}}-E)^n.\overline{\mathcal{L}}\big) \ge 0$
	for any flag ideal $\mathcal{J}$ {\color{black}{which is not of the form $(t^N)$}}. 
	The equality holds if and only if $\dim \mathrm{Supp}(\mathcal{O}_{X\times \mathbb{A}^1}/\mathcal{J})=0$.
\end{Prop}

\subsection{Estimation of the Second term via Seshadri constants}

To get the positivity of the second term of (\ref{eq:decomposition}), we will show that it suffices to have the upper bounds of Seshadri constant of $\mathcal{J}$. 
Indeed, we have
\begin{Prop}\label{prop:second}
If there exists a positive constant $\varepsilon$ such that
\begin{equation}\label{eq:div_estimate}
	\bigg(\frac{n+1}{n}\bigg)K_{\mathcal{B}/X\times \mathbb{A}^1}
	-\mathrm{Sesh}(\mathcal{J}, (X\times \mathbb{A}^1, 
\mathcal{O}_{X\times \mathbb{A}^{1}}(-K_{X\times \mathbb{A}^1})))E
	>\varepsilon E,
\end{equation}
then the second term of (\ref{eq:decomposition}) is positive so that  $\DF(\mathcal{B},\mathcal{L}(-E))>0$.
If the left hand of (\ref{eq:div_estimate})
is effective $($possibly zero$)$, then the second term of (\ref{eq:decomposition}) is non-negative so that $\DF(\mathcal{B},\mathcal{L}(-E))\ge0$. 
\end{Prop}
\begin{proof}[proof of Proposition \ref{prop:second}] 
We have already seen that the first term of (\ref{eq:decomposition}) is non-negative. 
For the estimation of the second term, the following positivity is crucial.
\begin{Lem}[{\cite[Equation (3)]{OdII}}]\label{eq:positivity_intersection}
	$((\overline{\mathcal{L}}-E)^n.E)>0. $
\end{Lem}
Once we get Lemma \ref{eq:positivity_intersection}, (\ref{eq:div_estimate}) and Lemma \ref{eq:positivity_intersection} immediately imply that the second term of (\ref{eq:decomposition}) is strictly positive.
The rest of this subsection will be devoted to the proof of Lemma \ref{eq:positivity_intersection}.
We prepare the following two results.
\begin{Lem}[{\cite[Lemma 2.7]{OdII}}]\label{Lem:positivity2}
Assume $n\ge 2$. 
Then the following hold. 
$(\mathrm{i})$
We have the following equality of polynomials;
$$
	(T-1)^n(T+n)=T^{n+1}-\sum_{i=1}^
	{n}
	(n+1-i)(T-1)^{n-i}T^{i-1}.
$$
$(\mathrm{ii})$
The polynomials $(T-1)^{n-i}T^{i-1}$ for $1\le i \le n$ are linearly independent over $\mathbb{Q}$.
In particular, for the monomial $T^s$ for an arbitrary integer $s$ with $0<s \le n$, there exist intergers $m_i$ $(1\le i \le n)$ such that
$$
	T^s=\sum_{i=1}^{n}m_i(T-1)^{n-i}T^{i-1}.
$$
\end{Lem}
This is an elementary lemma on polynomials as Lemma \ref{Lem:positivity}, so we leave the proof to the reader.
\begin{Lem}[{\cite[{Lemma 2.8}]{OdII}}]\label{Lem:intersection_number}
$(\mathrm{i})$
For any $1\le i \le n-1$,
\begin{equation}\label{eq:intersection1}
	(-E^2.(\overline{\mathcal{L}}-E)^{n-1}.(\overline{\mathcal{L}})^{i-1})\ge0.
\end{equation}
$(\mathrm{ii})$
Let $s=\dim(\mathrm{Supp}(\mathcal{O}_{X\times \mathbb{A}^1}/\mathcal{J}))$.
\begin{equation}\label{eq:intersection2}
	((-E)^{n+1-s}.(\overline{\mathcal{L}})^{s})<0.
\end{equation}
\end{Lem}
\begin{proof}[proof of Lemma \ref{Lem:intersection_number}]
By cutting $X\times \mathbb{P}^1$ by the divisors corresponding to $\overline{\mathcal{L}}^{\otimes r}$ and $(\overline{\mathcal{L}}-E)^{\otimes r}$, the proof of (\ref{eq:intersection1}) (resp. (\ref{eq:intersection2})) can be reduced to the case where $\dim(X)=2$ (resp. $\dim(X)=n+1-s$).
Then, (\ref{eq:intersection1}) (resp. (\ref{eq:intersection2})) follows from the Hodge index theorem (resp. the relative ampleness of $(-E)$).
The proof of Lemma \ref{Lem:intersection_number} is completed.
\end{proof}
\begin{proof}[proof of Lemma \ref{eq:positivity_intersection}] 
We decompose the left hand of Lemma \ref{eq:positivity_intersection}
as follows;
\begin{equation}\label{eq:decomposition1}
	(n+1)((\overline{\mathcal{L}}-E)^n.E)
	=
	((\overline{\mathcal{L}}-E)^n.(\overline{\mathcal{L}}+nE))
	-(\overline{\mathcal{L}}-E)^{n+1}.
\end{equation}
The second term in (\ref{eq:decomposition1}) is non-positive, in fact
\begin{eqnarray}
	\nonumber
		(\overline{\mathcal{L}}-E)^{n+1}
	&=&
		(\overline{\mathcal{L}}-E)^{n+1}-(\overline{\mathcal{L}})^{n+1}
	\\
	\label{eq:decomposition2}
	&=&
		-(E.\sum_{i=0}^{n}(\overline{\mathcal{L}}-E)^i.(\overline{\mathcal{L}})^{n-i})
		\le 0.
\end{eqnarray}
Let us apply Lemma \ref{Lem:positivity2} to the first term in (\ref{eq:decomposition1}).
Denote  $\dim(\mathrm{Supp}(\mathcal{O}_{X\times \mathbb{A}^1}/\mathcal{J}))$ by $s$. 
Let consider the case where $n\ge2$. 
From Lemma \ref{Lem:positivity2}, we find that for a sufficiently small $\varepsilon'>0$ there exist (small) real constants $\varepsilon_i$ $(1\le i \le n)$ such that
\begin{eqnarray}
		\nonumber
			((\overline{\mathcal{L}}-E)^n.(\overline{\mathcal{L}}+nE))
		&=&
			(\overline{\mathcal{L}})^{n+1}
			+
			(-E^2.
			\sum_{i=1}^{n}(n+1-i)
			((\overline{\mathcal{L}}-E)^{n-i}.(\overline{\mathcal{L}})^{i-1}))
		\\
		\nonumber
		&=&
			(-E^2.
			\sum_{i=1}^{n}(n+1-i)(\overline{\mathcal{L}}-E)^{n-i}.
			(\overline{\mathcal{L}})^{i-1})	
		\\
		\nonumber
		&=&
			\sum_{i=1}^{n}(n+1-i+\varepsilon_i)(-E^2.(\overline{\mathcal{L}}-E)^{n-i}.
			(\overline{\mathcal{L}})^{i-1})
		\\
		\label{eq:decomposition3}
		&&\qquad
			-\varepsilon'((-E)^{n+1-s}.(\overline{\mathcal{L}})^{s})
\end{eqnarray}
and $n+1-i+\varepsilon_i>0$ for all $1\le i \le n-1$.
From Lemma \ref{Lem:intersection_number} and (\ref{eq:decomposition3}), we find that   the first term in (\ref{eq:decomposition1}) is strictly positive.
This holds for the case where $n=1$ too, because
$$
	((\overline{\mathcal{L}}-E).(\overline{\mathcal{L}}+E))
	=-(E.E)>0.
$$
The proof of Lemma \ref{eq:positivity_intersection} is completed.
\end{proof}
\end{proof}

\section{Proofs}\label{proof}
Now, we complete the proof of theorems and corollaries.
\begin{proof}[proof of Theorem \ref{Thm:main}]
From Proposition \ref{prop:first}, we get
\begin{equation}\label{sesh.bound}
	\mathrm{Sesh}(\mathcal{J}, (X\times \mathbb{A}^1, 
\mathcal{O}_{X\times \mathbb{A}^1}(-K_{X\times \mathbb{A}^1})))
	<
	\biggl(
		\frac{n+1}{n}
	\biggr)
		\min_i 
	\bigg\{
		 \frac{a_i-b_i+1}{c_i}
	\bigg\}
\end{equation}
for any flag ideal $\mathcal{J}$.
Then,
\begin{eqnarray*}
	&&
		\bigg(
			\frac{n+1}{n}
		\bigg)
			K_{\mathcal{B}/X\times \mathbb{A}^1}
		-\mathrm{Sesh}
		(\mathcal{J}, (X\times \mathbb{A}^1, 
\mathcal{O}_{X\times \mathbb{A}^{1}}(-K_{X\times \mathbb{A}^1})))E
	\\
	&&
		\qquad
		>
		\bigg(
			\frac{n+1}{n}
		\bigg)
			\sum a_i E_i
		-
		\bigg(
			\frac{n+1}{n}
		\bigg)
			\min_i 
		\bigg\{
			\frac{a_i-b_i+1}{c_i}
		\bigg\}
			\sum c_iE_i
	\\
	&&
		\qquad
		=
		\bigg(
			\frac{n+1}{n}
		\bigg)
		\sum \bigg\{
			\bigg(
				\frac{a_i-b_i+1}{c_i}-\min_i\big\{\frac{a_i-b_i+1}{c_i}
			\bigg\}\bigg)
			+\frac{b_i-1}{c_i}
		\bigg\}
		c_iE_i
	\\
	&&
		\qquad
		\ge
		0.
\end{eqnarray*}
The proof is completed due to Proposition \ref{prop:second}.
\end{proof}

We comment on the case $\lct_{G}(X)=\frac{n}{n+1}$. 
From the proof of Proposition \ref{prop:positivity_first_term}, we find that if $X$ is not K-stable under the assumption, then $\dim\Supp(\mathcal{O}_{X\times \mathbb{A}^1}/\mathcal{J})$ should be zero. 
Such situation seems to be quite rare as partially proved in \cite[Theorem 4.29]{RT07} and \cite[proof of Theorem 4.1]{PR07}. 
Let us assume that $X$ is smooth. 
Let us recall that the minimal discrepancy of a smooth closed point in $X\times \mathbb{A}^{1}$ is $n$ (cf.\ e.\ g.\ \cite{Amb06}). On the other hand, 
$\Sesh(\mathcal{J},(X\times \mathbb{A}^{1}, 
\mathcal{O}_{X\times \mathbb{A}^{1}}(-K_{X\times \mathbb{A}^{1}})))$ is at most $n$ 
if $X$ is not isomorphic to $\mathbb{P}^{n}$, because 
$$
	n\ge\Sesh(\mathfrak{m}_{x,X},(X,\mathcal{O}_{X}(-K_X)))
$$ 
by \cite[Theorem 1.7]{BS09} and 
\begin{eqnarray}
	\nonumber
		\Sesh(\mathfrak{m}_{x,X},(X,\mathcal{O}_{X}(-K_X)))
	&\geq&
		 \Sesh(I_0,(X,\mathcal{O}_{X}(-K_X)))
	\\
	\nonumber
	&\geq& 
		\mathrm{Sesh}(\mathcal{J}, (X\times \mathbb{A}^1, 
\mathcal{O}_{X\times \mathbb{A}^{1}}(-K_{X\times \mathbb{A}^1})))
\end{eqnarray}
by the condition that $s=0$ (cf.\ e.\ g.\ \cite[Lemma 4.7]{PR07}).
Here, $\mathfrak{m}_{x,X}$ is the maximal ideal of $\mathcal{O}_{X,x}$.
Let us recall that we proved 
\begin{equation}\label{lct.girigiri} 
	\bigg(\frac{n+1}{n}\bigg)K_{\mathcal{B}/X\times \mathbb{A}^1}
	-\mathrm{Sesh}(\mathcal{J}, (X\times \mathbb{A}^1, 
\mathcal{O}_{X\times \mathbb{A}^{1}}(-K_{X\times \mathbb{A}^1})))E 
	\geq 0.  
\end{equation}
From the above three remarks, 
it is likely that we could strengthen the inequality (\ref{lct.girigiri}) 
so that the corresponding Donaldson-Futaki invariants are positive. 
Hence, we expect 
\begin{Conj}
Let $X$ be a smooth Fano manifold of dimension $n$. 
If $\lct(X)=\frac{n}{n+1}$ $($resp.\ $\lct_{G}(X)=\frac{n}{n+1}$$)$, then $(X,\mathcal{O}_{X}(-K_{X}))$ is K-stable $($resp.\ $G$-equivariantly K-stable$)$ 
or $X$ is isomorphic to $\mathbb{P}^{n}$. 
\end{Conj}

Corollary \ref{aut:fin} can be proved by Theorem \ref{Thm:main} and Matsushima's theorem.
The latter says that if a smooth Fano manifold $X$ admits K\"ahler-Einstein metrics, then $\mathrm{Aut}(X)$ is reductive.
\begin{proof}[proof of Corollary \ref{aut:fin}]
From Fact \ref{fact:tian}, $X$ admits K\"ahler-Einstein metrics.
Then, Matsushima's theorem implies that $\mathrm{Aut}(X)$ is reductive.
On the other hand, K-stability in Theorem \ref{Thm:main} implies that $\mathrm{Aut}(X)$ does not admit any non-trivial one-parameter subgroup $\mathbb{G}_m$.
Therefore, $\mathrm{Aut}(X)$ is finite by \cite[Chapter 4, Theorem 2.7]{Mil05}.
\end{proof}

The proofs of Theorem \ref{Thm:main_g} and Corollary \ref{aut:ss} are parallel to those of Theorem \ref{Thm:main} and of Corollary \ref{aut:fin}.  
In fact, from a $G$-equivariant test configuration $(\mathcal{X},\mathcal{L})$, we obtain a $G$-invariant flag ideal $\mathcal{J}=\sum I_{i}t^{i}$ whose blow up gives a resolution of indeterminacy of the natural rational map $\mathcal{X}\dashrightarrow X\times\mathbb{A}^{1}$. 
By interpreting $\lct((X\times \mathbb{A}^{1}, X\times \{0\});I_{0})$ as log canonical threshold for the corresponding sublinear system of $\mid -mK_{X} \mid$ by \cite[Example 9.2.23]{Laz04}, we obtain the upper bound for $\Sesh(I_0,(X,\mathcal{O}_{X}(-K_X)))$ similarly. 
The proof of Corollary \ref{aut:ss} uses the fact that a reductive algebraic group whose 
center does not have any nontrivial one parameter subgroup is semisimple (cf.  \cite[Chapter 1, Theorem 17.10]{Mil05}). 
We leave the detail to the reader. 

We note that we can also prove Corollary \ref{aut:ss} analytically. 
Let us fix a $G$-invariant K\"ahler-Einstein metric $\omega_{KE}$ and consider 
$$
F:=\{g\in \Aut^{0}(X)\mid (g^{-1})^{*}\omega_{KE} \mbox{ is $G$-invariant }\} \subset \Aut^{0}(X).  
$$
From the estimate by Tian, we know that the set of $G$-invariant K\"ahler-Einstein metrics is compact. Furthermore, $\Aut^{0}(X)$ acts transitively on the set 
of all K\"ahler-Einstein metrics, due to \cite{BM85}, and the isotropy of the  
action is compact (cf.\ \cite{Mat57}). Therefore, we conclude that $F$ is also compact. 

On the other hand, $F$ should contain the center $Z$ of $\Aut^{0}(X)$ which as a closed subset. On the other hand, its identity component 
is isomorphic to an algebraic torus (cf.\ \cite[Chapter1, Theorem 17.10]{Mil05}). 
Therefore, $Z$ should be discrete and we end the proof.



%% References


%\bibliographystyle{elsarticle-num}



\end{document}